\documentclass{article}
\usepackage{amsmath}
\usepackage{ntheorem-hyper}
\usepackage{amssymb}
\usepackage{authblk} 
\usepackage{algorithm}
\usepackage{algorithmicx}
\usepackage{algpseudocode}
\usepackage{graphicx}
\usepackage{graphbox}
\usepackage{subfigure}
\usepackage{svg}
\usepackage{epstopdf}
\usepackage{tabularx}  
\usepackage{booktabs}
\usepackage[numbers, sort]{natbib}
\usepackage{hyperref}
\usepackage[top=2cm, bottom=2cm, left=2cm, right=2cm]{geometry}  
\usepackage{xcolor} 
\usepackage{color}
\usepackage[title]{appendix}

\usepackage[T1,T2A]{fontenc}
\usepackage[utf8]{inputenc}
\usepackage[russian, german, english]{babel}

\hypersetup{urlcolor=blue, citecolor=red}

\newtheorem{theorem}{Theorem}
\newtheorem{prop}{Proposition}
\newtheorem*{proof}{Proof}
\newtheorem{remark}{Remark}[section]

\newtheorem{lemma}{Lemma}

\newtheorem{example}{Example}

\def\RR{\mathbb{R}}

\newcommand{\R}{\mathbb R}
\newcommand{\Rplus}{[0,\infty)}
\newcommand{\MaxCol}{\operatorname{MaxCol}}
\newcommand{\dist}{\operatorname{dist}}
\newcommand{\normmax}[1]{\lVert #1\rVert_{\max}}

\title{Example of exponentially many collisions of identical hard disks in the plane}
\author[+]{Hongda Qiu}
\begin{document}
\maketitle
\begin{abstract}
	Let $\MaxCol(n,d)$ denote the maximal number of nonsimultaneous elastic collisions among $n$ identical hard balls moving freely in $\R^d$. We give a planar construction with exponential lower bound in $n$:
	\[
	\MaxCol(n,2)\geq 2^{\lfloor n/2\rfloor},
	\qquad n\geq3.
	\]
\end{abstract}

\section{Introduction}



Consider a billiard system of $n$ identical hard balls of radius $1/2$ moving freely and colliding elastically in $\RR^d$. Write $q=(q_1,\ldots,q_n)\in\RR^{nd}$ for the center of the balls. Its configuration space, denoted by $B_{n,d}$, is defined as $\{q\in\RR^{nd}||q_i-q_j|\geq1,\quad i\neq j\}$. In other words, $B_{n,d}$ is the complement of all round cylinders $|q_i-q_j|<1$ with $i<j$. We refer to boundaries of these cylinders as walls. 

A trajectory of $n$ balls is a piecewise-linear curve running in $B_{n,d}$ and reflecting symmetrically on walls. If a trajectory $\gamma$ hits a wall of normal $\nu$ at time $t$, its elastic reflection can be described by
\begin{equation}
	\label{equation elastical collision}
	\gamma'(t_+)=\gamma'(t_-)-2\langle \gamma'(t_-),\nu \rangle\nu.
\end{equation}
For nonsimultaneous collisions, a trajectory never hit intersection of two or more walls.

At a point $q$ where several pairs of balls are in contact, the corresponding walls intersect. Their inward tangent half-spaces determine
the tangent cone $K$ of $B_{n,d}$ at $q$. A nonsingular billiard trajectory
in $K$ with finitely many collisions can be
realized by a trajectory in $B_{n,d}$ sufficiently
close to $q$, as shown by Burago and Ivanov
\cite[Lemma~2.2]{BI}. Thus a suitable contact configuration may give
rise to a hard-ball trajectory with many collisions.

Let $\MaxCol(n,d)$ denote the maximal number of nonsimultaneous collisions among $n$ identical hard balls moving freely in $\R^d$. Vaserstein and Galperin proved that every
hard-ball trajectory has finitely many collisions \cite{Va,Ga}.
Uniform bounds on $\MaxCol(n,d)$ were obtained
by Burago, Ferleger, and Kononenko \cite{BFK}
and later improved by Burdzy \cite{BUpper}.
The number can nevertheless grow rapidly with $n$.
Burdzy and Duarte constructed examples with
a cubic-order number of collisions in the plane
\cite{BD}, while Burago and Ivanov obtained
an exponential lower bound in dimension three
\cite{BI}.

In the three-dimensional construction of \cite{BI}, three contact directions issuing from the
same ball can be chosen nearly mutually
perpendicular. For example, four balls
centered at
\[
(0,0,0),\quad (1,0,0),\quad
(0,1,0),\quad (0,0,1)
\]
have three mutually perpendicular contact
directions at the first ball.

Such a configuration is impossible in the plane. Consider an alternative construction of four disks with three contacts as in Figure~\ref{fig:four_disk}. Denote the three contacts between balls by $e_1=[1,2],e_2=[2,3],e_3=[2,4]$. The directions of $e_2$ and $e_3$ are nearly opposite. With a suitable perturbation, we can make collisions occur along contacts in the order of
\[e_3,\quad e_2,\quad e_3,\quad e_1,\quad e_3,\quad e_2, \quad e_3\]
(Example~\ref{example four disks}). The central disk repeatedly collides with the two lower disks while colliding with the upper
disk only once. This pattern can be recursively extended to configurations with more disks. We obtain the following result.

\begin{figure}[htbp]
	\centering
	\includegraphics[width=0.25\textwidth]{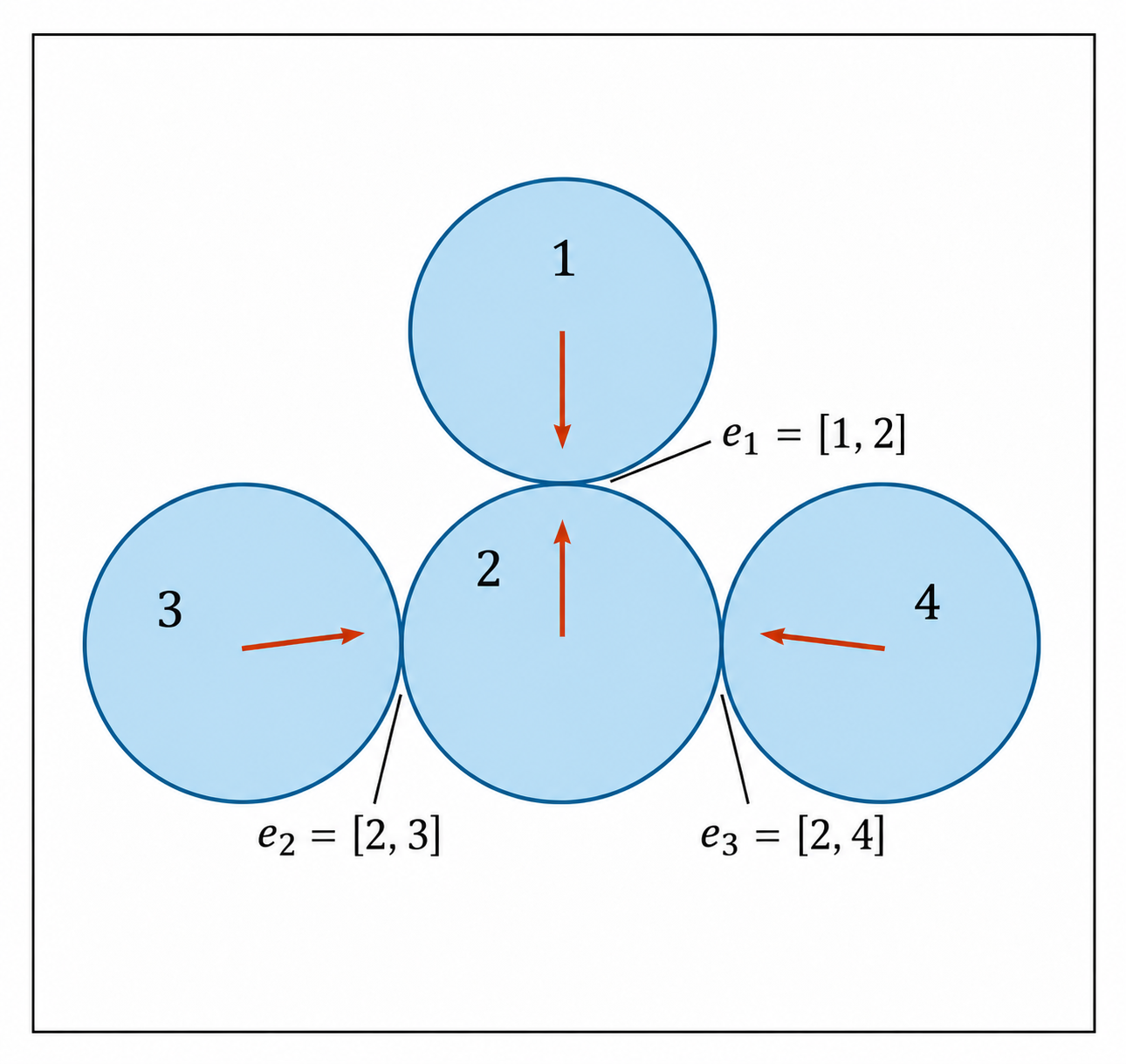}
	\caption{Four disks having three contacts with initial velocities represented by red arrows}
	\label{fig:four_disk}
\end{figure}


\begin{theorem}\label{thm:main}
	For every $n\geq3$,
	\[
	\MaxCol(n,2)\geq2^{\lfloor n/2\rfloor}.
	\]
\end{theorem}
Assume that there are $m$ contacts at $q$. Denote by $\nu_1,\ldots,\nu_m$ the inward unit normals of walls that form $K$. We construct our example in a way that $\nu_1,\ldots,\nu_m$ are linearly independent (Lemma~\ref{lem:independence}).
In such a case, $K$ is isometric to $K_0\times\RR^{nd-m}$, where
$K_0=K\cap\operatorname{span}\{\nu_1,\ldots,\nu_m\}$ is an
$m$-dimensional polyhedral cone with the same inward unit wall
normals. The linear factor $\RR^{nd-m}$ does not affect collisions, so it suffices to consider trajectories
in $K_0$.

A nonsingular billiard trajectory $\gamma:I\to K_0$ is uniquely
represented by $ t\longmapsto(f_1(t),\ldots,f_m(t)) $, where $f_i$ are nonnegative piecewise-linear functions defined as
\[
f_i(t)=\langle\gamma(t),\nu_i\rangle,
\qquad i=1,\ldots,m.
\]
One may think of $f_i(t)$ as the distance to the tangent plane to the
$i$-th wall at $q$. If we take the derivative and dot-product by $\nu_j$ for both sides of Equation~\ref{equation elastical collision}, we obtain
\begin{equation}
	\label{equation physical equation}
	f_j'(t_+) = f_j'(t_-)-2\langle \nu_i,\nu_j \rangle f_i'(t_-),\quad j=1,...,m.
\end{equation}
Thus the reflection law in wall coordinates
is determined by the Gram matrix
$(\langle\nu_i,\nu_j\rangle)_{i,j=1}^m$.

We therefore
consider more general reflection rules obtained by
replacing these scalar products with the entries
of an admissible matrix. We adopt some language introduced by Burago and Ivanov \cite{BI}. An $m\times m$ matrix $A=(a_{ij})$ is called \emph{admissible} if
\[
a_{ii}=1
\qquad\text{for every }i.
\]
Note that an admissible matrix does not have to be symmetric. 

An $A$-trajectory is a piecewise-linear map
\[
f=(f_1,\dots,f_m):I\longrightarrow \Rplus^m
\]
having finitely many break points and no simultaneous zeros, such that (i) $f$ is linear if all $f_i>0$,  (ii) if $f_i(t)=0$, then
\[
f_j'(t+)=f_j'(t-)-2a_{ij}f_i'(t-),
\qquad j=1,\dots,m,
\]
and (iii) no collisions occur at end points of $I$.

A generalized $A$-trajectory is defined similarly, except that several coordinates may vanish simultaneously, provided that
$a_{ij}=a_{ji}=0$ whenever $f_i(t)=f_j(t)=0$ for $i\neq j$. At such a time, the reflection rule becomes
\[
f_j'(t+)=f_j'(t-)
-2\sum_{i:f_i(t)=0}a_{ij}f_i'(t-),
\qquad j=1,\ldots,m.
\]
A simultaneous zero of $k$ coordinates counts
as $k$ collisions.
See \cite[Definitions~3.2 and~3.6]{BI}.


To pass from generalized trajectories to actual hard-ball trajectories, we use the following results of Burago and Ivanov \cite{BI}.
\begin{lemma}[Burago--Ivanov]\label{lem:BI-tools}
	\begingroup
	\renewcommand{\labelenumi}{(\roman{enumi})}
	\renewcommand{\theenumi}{\thelemma(\roman{enumi})}
	\begin{enumerate}
		
		\item\label{lemma 1}
		A generalized $A$-trajectory with $N$ collisions can be perturbed to a genuine $A$-trajectory with $N$ collisions \cite[Lemma~3.7]{BI}.
		
		\item\label{lemma 2}
		If an admissible matrix $A$ admits an $A$-trajectory with $N$ collisions, then every admissible matrix $\tilde{A}$ that is close enough to $A$ (in matrix max norm) admits a $\tilde{A}$-trajectory with $N$ collisions
		\cite[Lemma~3.3]{BI}.
		
		\item\label{lemma 3}
		Given positive numbers $\lambda_1,\dots,\lambda_m$, define
		\[
		(A^\lambda)_{ij}=\frac{\lambda_j}{\lambda_i}a_{ij}.
		\]
		If $A$ admits an $A$-trajectory with $N$ collisions, then so does $A^\lambda$
		\cite[Lemma~3.4]{BI}.
		
		\item\label{lemma 4}
		If the tangent cone at a hard-ball configuration admits a billiard trajectory with $N$ collisions, then an actual hard-ball system admits a trajectory with at least $N$ collisions \cite[Lemma~2.2]{BI}.
	\end{enumerate}
	\endgroup
\end{lemma}

For the rest of content, we focus on disks in the plane and fix $d=2$. We begin by constructing a model trajectory with $m$ initial contacts leading to exponentially many collisions in Section~\ref{section model trajectory}, then realize its dynamics through a planar configuration in Section~\ref{section planar caterpillar packing}. Finally, we construct an actual trajectory with the same amount of collisions as the model trajectory in Section~\ref{section proof of main theorem}. This proves our main theorem. 

\textit{Acknowledgement.} I would like to thank Dmitri Burago and Anton Petrunin for their time on reading the first draft of this work and for their encouragement.

%

\section{Model trajectory}
\label{section model trajectory}
For $m\geq1$, define an admissible matrix
\[
C_m=(c_{ij})_{i,j=1}^m
\]
by
\[
c_{ij}=
\begin{cases}
	1,&i=j,\\
	-1,&i\text{ is odd and }j\in\{i+1,i+2\},\\
	-\dfrac12,&|i-j|=1\text{ and }\min\{i,j\}\text{ is even},\\
	0,&\text{otherwise}.
\end{cases}
\]
Thus, for
\[
p=2k-1,
\qquad q=2k,
\qquad r=2k+1,
\]
the block on the indices $p,q,r$ is
\[
\left.C_m\right|_{\{p,q,r\}}
=
\begin{pmatrix}
	1&-1&-1\\
	0&1&-\frac12\\
	0&-\frac12&1
\end{pmatrix},
\]
whenever $r\leq m$. In particular, $C_m$ is generally not symmetric.

\begin{lemma}\label{lem:model-trajectory}
	For every $m\geq1$, the matrix $C_m$ admits a generalized trajectory with at least $ 2^{\lfloor(m+1)/2\rfloor} $ collisions.
\end{lemma}

\begin{proof}
	We construct the zero sets of coordinate functions recursively. Set
	\[
	T_1=\{0\},
	\]
	and regard $T_1$ as an arithmetic progression with common difference $1$.
	
	Suppose that, for some $k\geq1$,
	\[
	p=2k-1
	\]
	and
	\[
	T_p=\{x_1<\cdots<x_M\}
	\]
	is an arithmetic progression with common difference $\beta>0$. Put $q=2k$ and define
	\[
	y_s=x_1+\left(s-\frac32\right)\beta,
	\qquad 1\leq s\leq M+1,
	\]
	\[
	T_q=\{y_1,\dots,y_{M+1}\}.
	\]
	Then
	\[
	y_1<x_1<y_2<x_2<\cdots<y_M<x_M<y_{M+1},
	\]
	so $T_p\cup T_q$ is an arithmetic progression with common difference $\beta/2$.
	
	If $r=q+1\leq m$, define
	\[
	z_\ell=y_1+\left(\ell-\frac32\right)\frac{\beta}{2},
	\qquad 1\leq\ell\leq2M+2,
	\]
	\[
	T_r=\{z_1,\dots,z_{2M+2}\}.
	\]
	The points of $T_p\cup T_q$ are precisely the midpoints between consecutive points of $T_r$.
	
	For every odd $i$, define
	\[
	f_i(t)=\dist(t,T_i).
	\]
	If $m$ is even, define the terminal even component by
	\[
	f_m(t)=\dist(t,T_m).
	\]
	
	It remains to define an even component $f_q$ when $q+1\leq m$. On each interval $[y_s,y_{s+1}]$, set
	\[
	f_q(t)=
	\begin{cases}
		2(t-y_s),
		& y_s\leq t\leq y_s+\frac{\beta}{4},\\[1mm]
		t-y_s+\frac{\beta}{4},
		& y_s+\frac{\beta}{4}\leq t\leq y_s+\frac{\beta}{2},\\[1mm]
		\frac{5\beta}{4}-(t-y_s),
		& y_s+\frac{\beta}{2}\leq t\leq y_s+\frac{3\beta}{4},\\[1mm]
		2\beta-2(t-y_s),
		& y_s+\frac{3\beta}{4}\leq t\leq y_{s+1}.
	\end{cases}
	\]

	Thus the successive slopes of $f_q$ on $[y_s,y_{s+1}]$ are respectively
	\[
	2,
	\qquad1,
	\qquad-1,
	\qquad-2.
	\]
	To the left of $y_1$, continue $f_q$ with slope $-2$ on $[y_1-\beta/4,y_1]$ and slope $-1$ on $(-\infty,y_1-\beta/4]$. To the right of $y_{M+1}$, continue it with slope $2$ on $[y_{M+1},y_{M+1}+\beta/4]$ and slope $1$ on $[y_{M+1}+\beta/4,\infty)$. The resulting function is continuous, nonnegative, and has zero set exactly $T_q$.
	
	We now verify the generalized reflection law. It is convenient to do so coordinate by coordinate.
	
	Let $j$ be odd. The breakpoints of $f_j=\dist(\,\cdot\,,T_j)$ are its roots $T_j$ and, when $j\geq3$, the midpoints between consecutive roots. By construction, those midpoints are exactly the points of
	\[
	T_{j-2}\cup T_{j-1}.
	\]
	At a point of $T_j$, the slope changes from $-1$ to $1$, which is the self-reflection corresponding to $c_{jj}=1$. At a point of $T_{j-2}$, the slope changes from $1$ to $-1$; since
	\[
	c_{j-2,j}=-1
	\qquad\text{and}\qquad
	f_{j-2}'(t-)=-1,
	\]
	the required jump is
	\[
	-2c_{j-2,j}f_{j-2}'(t-)=-2.
	\]
	At a point of $T_{j-1}$, the same slope change is produced by
	\[
	c_{j-1,j}=-\frac12,
	\qquad
	f_{j-1}'(t-)=-2,
	\]
	again giving the jump $-2$. No other row of $C_m$ has a nonzero entry in column $j$.
	
	Let $j$ be an internal even index, so $j+1\leq m$. The breakpoints of $f_j$ are exactly the points of
	\[
	T_{j-1}\cup T_j\cup T_{j+1}.
	\]
	At a point of $T_j$, the slope changes from $-2$ to $2$, as required by $c_{jj}=1$. At a point of $T_{j-1}$, it changes from $1$ to $-1$; this equals the jump
	\[
	-2c_{j-1,j}f_{j-1}'(t-)
	=-2(-1)(-1)=-2.
	\]
	At a point of $T_{j+1}$, the slope decreases by $1$, namely
	\[
	2\longmapsto1
	\qquad\text{or}\qquad
	-1\longmapsto-2.
	\]
	Since
	\[
	c_{j+1,j}=-\frac12
	\qquad\text{and}\qquad
	f_{j+1}'(t-)=-1,
	\]
	the required jump is
	\[
	-2c_{j+1,j}f_{j+1}'(t-)=-1.
	\]
	Again, no other row has a nonzero entry in column $j$.
	
	Finally, if $j=m$ is even, then $f_j=\dist(\,\cdot\,,T_j)$. Its roots give the self-reflection $-1\mapsto1$, and its midpoints are precisely the points of $T_{j-1}$, where the jump $1\mapsto-1$ agrees with $c_{j-1,j}=-1$.
	
	It follows that every breakpoint occurs at a time when at least one component vanishes and that the generalized jump formula holds at every such time.
	
	Indeed, every pair joined by a nonzero off-diagonal
	entry belongs either to a complete triple $\{p,q,r\}$
	or, when $m$ is even, to the terminal pair
	$\{m-1,m\}$. The zero sets associated with each
	complete triple are pairwise disjoint, and the terminal zero sets
	$T_{m-1}=\{x_1,\ldots,x_M\}$ and
	$T_m=\{y_1,\ldots,y_{M+1}\}$ satisfy
	\[
	y_1<x_1<y_2<x_2<\cdots<y_M<x_M<y_{M+1},
	\]
	so $T_{m-1}\cap T_m=\varnothing$.
	Hence any two components that vanish simultaneously
	have zero matrix entries in both directions. Since all root sets are finite, we may restrict the trajectory to a bounded interval containing every root and having no collision at either endpoint. Thus the functions $f_1,\dots,f_m$ form a generalized $C_m$-trajectory.
	
	It remains to count the collisions. Denote
	\[
	a_i:=|T_i|.
	\]
	The construction gives
	\[
	a_1=1,
	\qquad
	a_{2k}=a_{2k-1}+1,
	\qquad
	a_{2k+1}=2a_{2k-1}+2.
	\]
	Hence
	\[
	a_{2k-1}=3\cdot2^{k-1}-2,
	\qquad
	a_{2k}=3\cdot2^{k-1}-1.
	\]
	Summing these cardinalities yields
	\[
	\sum_{i=1}^{2k-1}a_i
	=2^{k+2}+2^{k-1}-3k-5
	\]
	and
	\[
	\sum_{i=1}^{2k}a_i
	=2^{k+2}+2^{k+1}-3k-6.
	\]
	Therefore the total amount of collisions is
	\[
	N_m=
	\begin{cases}
		2^{k+2}+2^{k-1}-3k-5, & m=2k-1,\\[2mm]
		2^{k+2}+2^{k+1}-3k-6, & m=2k.
	\end{cases}
	\]
\end{proof}

The collision count agrees with \cite[Equation~$3.10$]{BI}.

Note that the collision times of the generalized
$C_m$-trajectory constructed in
Lemma~\ref{lem:model-trajectory} lie in
$(-3/2,3/2)$, independently of $m$.
Indeed, at the $k$-th step of the recursion,
the outermost collision times extend by at most
$3\beta_k/4$ on each side, where
$\beta_k=2^{-(k-1)}$, and
\[
\sum_{k=1}^{\infty}\frac{3\beta_k}{4}
=\frac32.
\]

For illustration, we present an example with $4$ disks and $3$ initial contacts; namely, an ``auxiliary" trajectory corresponding to $C_3$. 
\begin{example}[Four disks]
	\label{example four disks}
	For $n=4$, we have $m=3$. In
	Lemma~\ref{lem:model-trajectory}, take
	\[
	p=1,\qquad q=2,\qquad r=3,
	\qquad M=1,\qquad x_1=0,\qquad\beta=1.
	\]
	The admissible matrix is
	\[
	C_3 = \begin{pmatrix}
		1&-1&-1\\
		0&1&-\frac{1}{2}\\
		0&-\frac{1}{2}&1
	\end{pmatrix},
	\]
	where the two entries equal to $-\frac{1}{2}$ are corresponding to the two nearly opposite contacts $e_2$ and $e_3$ (see also Figure~\ref{fig:four_disk_contact_graph}).
	\begin{figure}[htbp]
		\centering
		\includegraphics[width=0.75\textwidth]{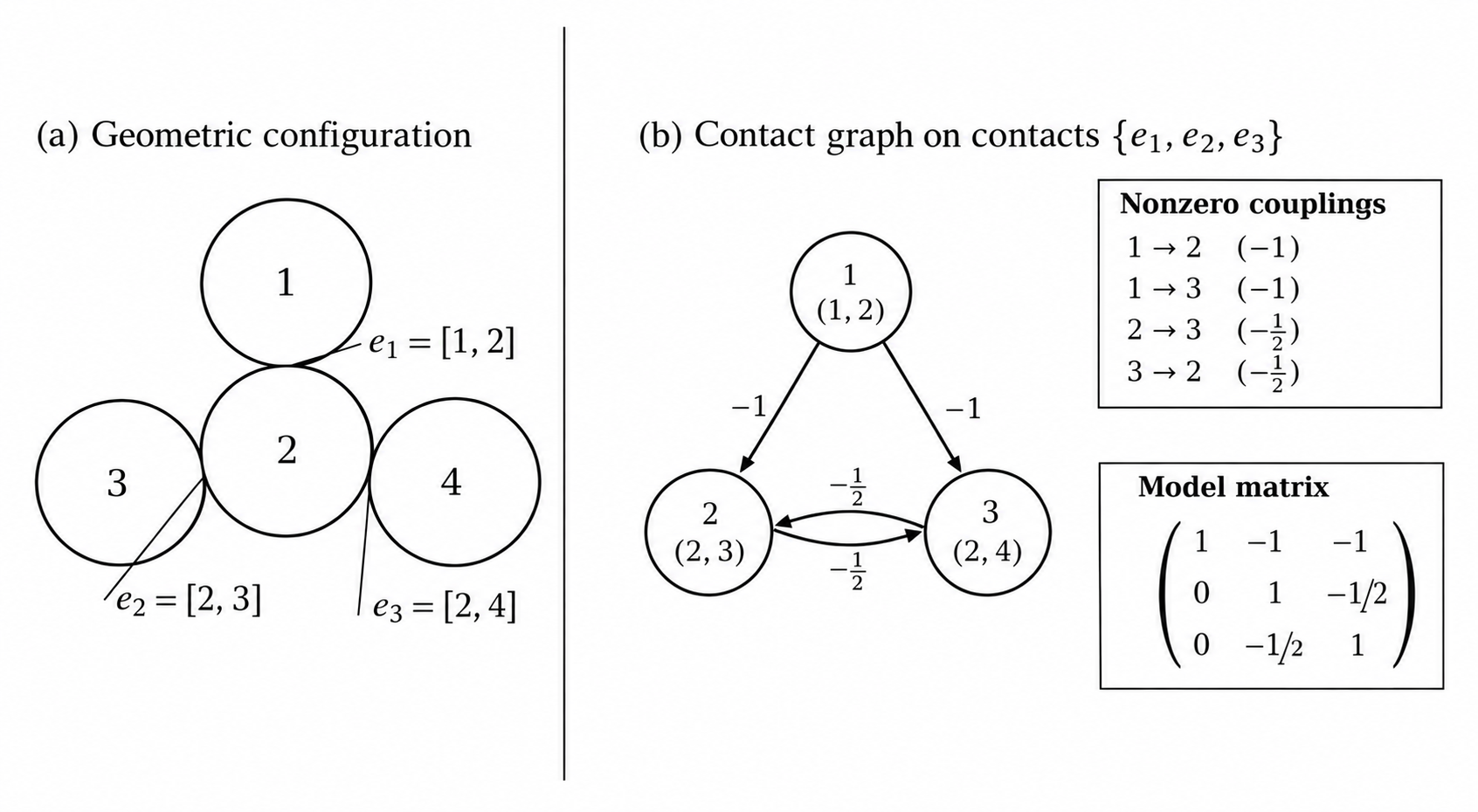}
		\caption{Four-disk structure in this work.}
		\label{fig:four_disk_contact_graph}
	\end{figure}
	
	The collision moments in model time are given by
	\[
	\begin{aligned}
		T_1&=\{x_1\}=\{0\},\\
		T_2&=\{y_1,y_2\}
		=\left\{-\frac12,\frac12\right\},\\
		T_3&=\{z_1,z_2,z_3,z_4\}
		=\left\{-\frac34,-\frac14,\frac14,\frac34\right\}.
	\end{aligned}
	\]
	They happen in the order $	z_1<y_1<z_2<x_1<z_3<y_2<z_4 $.
	
	Denote four disks as $q_1,q_2,q_3,q_4$, and their contacts as $e_1=\{q_1,q_2\},\quad
	e_2=\{q_2,q_3\},\quad
	e_3=\{q_2,q_4\}$. The trajectory is represented by $t\mapsto(f_1(t),f_2(t),f_3(t))$. Consequently, the chronological collision list is
	\[
	\begin{array}{c|c|c|c}
		\text{Order}
		& \text{Model time}
		& \text{Vanishing coordinate}
		& \text{Disk pair}\\
		\hline
		1 & z_1=-\frac34 & f_3 & (2,4)\\
		2 & y_1=-\frac12 & f_2 & (2,3)\\
		3 & z_2=-\frac14 & f_3 & (2,4)\\
		4 & x_1=0        & f_1 & (1,2)\\
		5 & z_3=\frac14  & f_3 & (2,4)\\
		6 & y_2=\frac12  & f_2 & (2,3)\\
		7 & z_4=\frac34  & f_3 & (2,4)
	\end{array}
	\]
	Equivalently, the contact sequence is
	\[
	e_3\longrightarrow e_2\longrightarrow e_3
	\longrightarrow e_1\longrightarrow e_3
	\longrightarrow e_2\longrightarrow e_3.
	\]
	There are exactly $	N_3=|T_1|+|T_2|+|T_3|=1+2+4=7 $ collisions. 
\end{example}

\section{Planar packing and compatibility}
\label{section planar caterpillar packing}
We next construct a planar configuration of
$m+1$ disks whose contact graph is compatible
with the auxiliary dynamics of Section 2.
The main geometric requirement is that the
prescribed contacts occur without introducing
any additional contacts.
\begin{figure}[!htbp]
	\centering
	\includegraphics[width=0.72\textwidth]{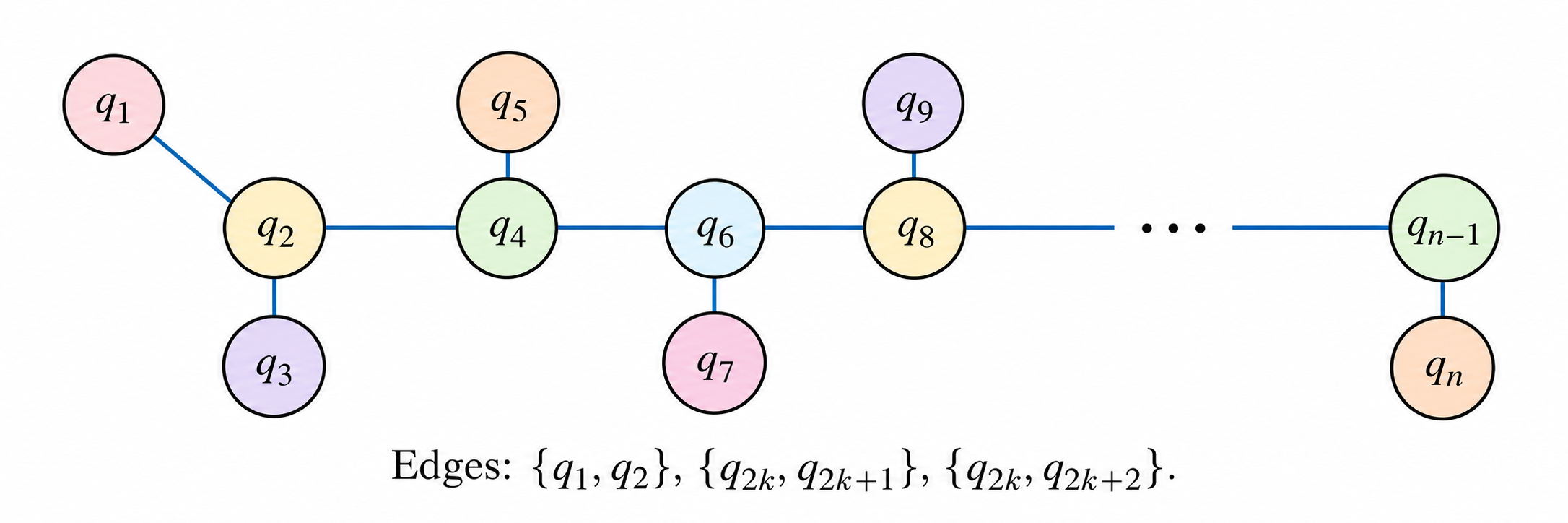}
	\caption{The contact graph of the planar caterpillar packing. The prescribed contacts are $\{q_1,q_2\}$, $\{q_{2k},q_{2k+1}\}$, and $\{q_{2k},q_{2k+2}\}$ whenever the displayed indices do not exceed $m+1$.}
	\label{fig:planar-caterpillar-contact-graph}
\end{figure}

\begin{lemma}\label{lem:packing}
	Let $m\geq1$ and $0<c\leq1/5$. There exists a configuration of $m+1$ points in $\R^2$ with contact graph
	\[
	\{1,2\},
	\qquad
	\{2k,2k+1\},
	\qquad
	\{2k,2k+2\},
	\]
	whenever the displayed indices do not exceed $m+1$, such that every prescribed contact has length $1$ and every noncontact pair has distance at least $1+2c$.
\end{lemma}

\begin{proof}
	Set
	\[
	s=\sqrt{1-c^2},
	\qquad
	\overrightarrow{d}=(s,-1-c).
	\]
	We construct the configuration by periodically translating a four-point pattern. For every $r\in\mathbb Z$, define
	\[
	\overrightarrow{X}_r=r\overrightarrow{d},
	\qquad
	\overrightarrow{Y}_r=r\overrightarrow{d}+(s,-c),
	\]
	\[
	\overrightarrow{Z}_r=r\overrightarrow{d}+(-s,-c),
	\qquad
	\overrightarrow{W}_r=r\overrightarrow{d}+\bigl((1+2c)s,\,1-c-2c^2\bigr).
	\]
	Enumerate the points by
	\[
	\overrightarrow{q}_1=\overrightarrow{Y}_{-1},
	\]
	and, for $r\geq0$,
	\[
	\overrightarrow{q}_{4r+2}=\overrightarrow{X}_r,
	\qquad
	\overrightarrow{q}_{4r+3}=\overrightarrow{Z}_r,
	\qquad
	\overrightarrow{q}_{4r+4}=\overrightarrow{Y}_r,
	\qquad
	\overrightarrow{q}_{4r+5}=\overrightarrow{W}_r,
	\]
	whenever the index is at most $m+1$.
	
	At $\overrightarrow{X}_r$, the three prescribed contact vectors are
	\[
	\overrightarrow{Y}_{r-1}-\overrightarrow{X}_r=(0,1),
	\qquad
	\overrightarrow{Z}_r-\overrightarrow{X}_r=(-s,-c),
	\qquad
	\overrightarrow{Y}_r-\overrightarrow{X}_r=(s,-c).
	\]
	At $\overrightarrow{Y}_r$, they are
	\[
	\overrightarrow{X}_r-\overrightarrow{Y}_r=(-s,c),
	\qquad
	\overrightarrow{W}_r-\overrightarrow{Y}_r=(2cs,1-2c^2),
	\qquad
	\overrightarrow{X}_{r+1}-\overrightarrow{Y}_r=(0,-1).
	\]
	All six vectors have length $1$. Hence every prescribed pair is tangent.
	
	We next exclude unintended contacts. For $\ell\in\mathbb Z$, direct calculation gives
	\[
	\begin{array}{c|c}
		\text{pair}&\text{squared distance}\\
		\hline
		\overrightarrow{X}_r,\overrightarrow{X}_{r+\ell};\ \overrightarrow{Y}_r,\overrightarrow{Y}_{r+\ell};\ \overrightarrow{Z}_r,\overrightarrow{Z}_{r+\ell};\ \overrightarrow{W}_r,\overrightarrow{W}_{r+\ell}
		&2(1+c)\ell^2\\[1mm]
		\overrightarrow{X}_r,\overrightarrow{Y}_{r+\ell}
		&1+2(1+c)\ell(\ell+1)\\[1mm]
		\overrightarrow{X}_r,\overrightarrow{Z}_{r+\ell};\ \overrightarrow{Y}_r,\overrightarrow{W}_{r+\ell}
		&1+2(1+c)\ell(\ell-1+2c)\\[1mm]
		\overrightarrow{X}_r,\overrightarrow{W}_{r+\ell}
		&2(1+c)(\ell^2+2c\ell+1)\\[1mm]
		\overrightarrow{Y}_r,\overrightarrow{Z}_{r+\ell}
		&2(1+c)\bigl(\ell^2-2(1-c)\ell+2(1-c)\bigr)\\[1mm]
		\overrightarrow{Z}_r,\overrightarrow{W}_{r+\ell}
		&5+8c-4c^2-8c^3+2(1+c)\ell(\ell+1).
	\end{array}
	\]
	The unit distances in the table are exactly
	\[
	\overrightarrow{X}_r\overrightarrow{Y}_{r-1},
	\qquad
	\overrightarrow{X}_r\overrightarrow{Y}_r,
	\qquad
	\overrightarrow{X}_r\overrightarrow{Z}_r,
	\qquad
	\overrightarrow{Y}_r\overrightarrow{W}_r,
	\]
	together with the equivalent translated contact $\overrightarrow{Y}_r\overrightarrow{X}_{r+1}$.
	
	We show that every remaining squared distance is at least $(1+2c)^2$. First,
	\[
	2(1+c)-(1+2c)^2=1-2c-4c^2>0
	\]
	for $0<c\leq1/5$. In the $\overrightarrow{X}_r,\overrightarrow{Y}_{r+\ell}$ row, the contact values are $\ell=-1,0$; for every other integer $\ell$,
	\[
	\ell(\ell+1)\geq2.
	\]
	In the $\overrightarrow{X}_r,\overrightarrow{Z}_{r+\ell}$ and $\overrightarrow{Y}_r,\overrightarrow{W}_{r+\ell}$ rows, the contact value is $\ell=0$, whereas for every nonzero integer $\ell$,
	\[
	\ell(\ell-1+2c)\geq2c.
	\]
	The quadratic factors in the $\overrightarrow{X}_r,\overrightarrow{W}_{r+\ell}$ and $\overrightarrow{Y}_r,\overrightarrow{Z}_{r+\ell}$ rows attain their integer minima at $\ell=0$ and $\ell=1$, respectively, and both minima equal $1$. Finally,
	\[
	\ell(\ell+1)\geq0
	\]
	for every integer $\ell$, and
	\[
	5+8c-4c^2-8c^3-(1+2c)^2
	=4(1+c)(1-2c^2)>0.
	\]
	Thus every noncontact pair has distance at least $1+2c$. Equality occurs, for example, for $\overrightarrow{X}_r$ and $\overrightarrow{Z}_{r+1}$.
\end{proof}

Let the contact edges be enumerated by
\[
e_1=\{q_1,q_2\},
\]
\[
e_{2k}=\{q_{2k},q_{2k+1}\},
\qquad
e_{2k+1}=\{q_{2k},q_{2k+2}\}.
\]

\begin{figure}[htbp]
	\centering
	\includegraphics[width=\textwidth]{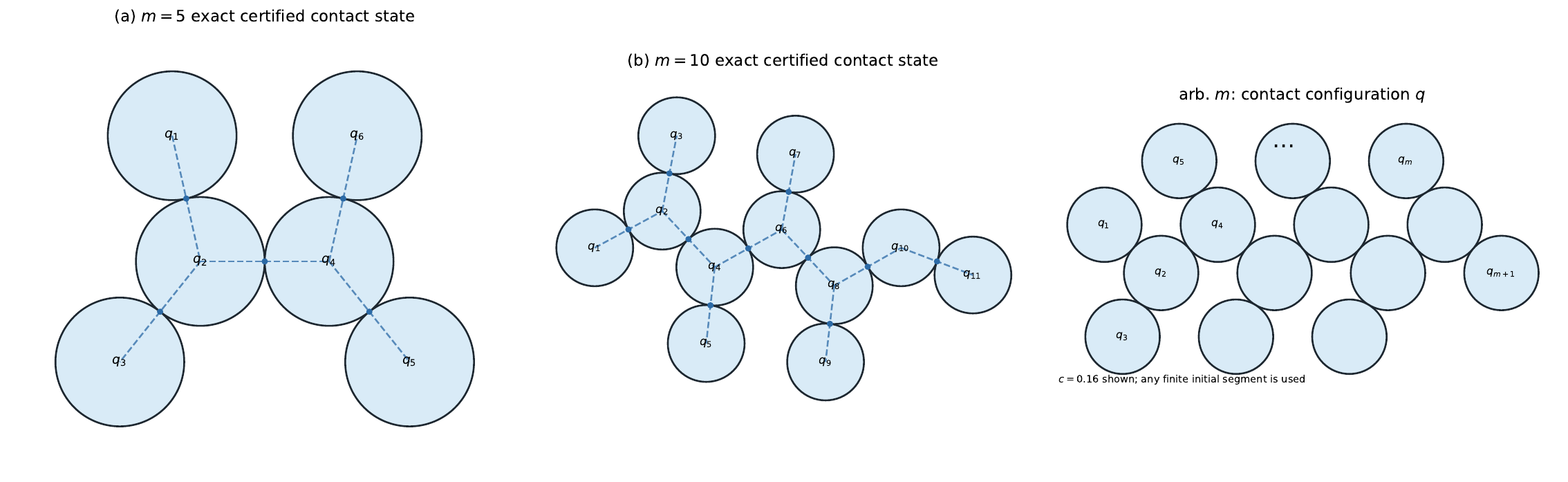}
	\caption{Reference contact configurations in their actual planar geometry, up to rigid motions. Panels~(a) and~(b) show exact finite certified configurations for $m=5$ and $m=10$. Panel~(c) shows a finite segment of the periodic caterpillar packing of Lemma~\ref{lem:packing}, with $c=0.16$ chosen only for visibility. The dashed center segments indicate the prescribed contact pairs, and the blue dots mark the corresponding tangency points.}
	\label{fig:contact-configurations}
\end{figure}


We now relate the matrix $C_m$ to
the Gram matrix of the wall normals of the
planar configuration constructed above.
Although the physical Gram matrix is symmetric,
the matrix $C_m$ need not be. We establish a connection between them
through a small perturbation and diagonal rescaling.

Let $B_m(\eta)$ be the Gram matrix of the unit wall normals in the tangent cone at the configuration of Lemma~\ref{lem:packing}, where
\[
\eta=\frac c2.
\]
At every trivalent disk, the three outward unit contact rays have pairwise scalar products
\[
-c,
\qquad
-c,
\qquad
-1+2c^2.
\]

If contacts \(\{i,j\}\) and \(\{j,k\}\) share disk \(j\), then $\langle\nu_{ij},\nu_{jk}\rangle = \frac12\langle q_i-q_j,q_k-q_j\rangle. $ If the contacts are disjoint, the corresponding wall normals are orthogonal. Therefore, for $p=2k-1$, $q=2k$, and $r=2k+1$,
\[
\left.B_m(\eta)\right|_{\{p,q,r\}}
=
\begin{pmatrix}
	1&-\eta&-\eta\\
	-\eta&1&-\frac12+4\eta^2\\
	-\eta&-\frac12+4\eta^2&1
\end{pmatrix}.
\]
Note that if $m=2k$ is even, the terminal pair of contacts
$\{m-1,m\}$ forms an incomplete block, with
\[
\left.B_m(\eta)\right|_{\{m-1,m\}}
=
\begin{pmatrix}
	1&-\eta\\
	-\eta&1
\end{pmatrix}.
\]
All off-diagonal entries not covered by these blocks vanish.

\begin{lemma}\label{lem:independence}
	The wall normals of the configuration in Lemma~\ref{lem:packing} are linearly independent. Consequently, $B_m(\eta)$ is positive definite.
\end{lemma}

\begin{proof}
	The contact graph is a tree. Suppose that
	\[
	\sum_e a_e\overrightarrow{\nu}_e=0,
	\]
	where $\overrightarrow{\nu}_e$ is the unit configuration-space normal associated to the contact $e$. Choose a leaf disk. In the two-dimensional coordinate block corresponding to that disk, exactly one normal $\nu_e$ is nonzero. Hence $a_e=0$. Remove that edge and its leaf and repeat. Induction over the number of edges shows that all coefficients vanish.
\end{proof}

\begin{proposition}\label{prop:convergence}
	For every $m\geq1$ and $0<\eta\leq1/10$, there exists an admissible matrix $A_{m,\eta}$ such that
	\[
	B_m(\eta)=A_{m,\eta}^{\lambda}
	\]
	and
	\[
	\normmax{A_{m,\eta}-C_m}\leq4\eta^2.
	\]
\end{proposition}

\begin{proof}
	Define
	\[
	\lambda_i=\eta^{\lfloor i/2\rfloor},
	\qquad 1\leq i\leq m,
	\]
	and set
	\[
	(A_{m,\eta})_{ij}
	=
	\frac{\lambda_i}{\lambda_j}(B_m(\eta))_{ij}.
	\]
	Then, by definition,
	\[
	B_m(\eta)=A_{m,\eta}^{\lambda},
	\qquad
	(A^\lambda)_{ij}=\frac{\lambda_j}{\lambda_i}a_{ij}.
	\]
	For $ p=2k-1, \quad q=2k, \quad r=2k+1$, we have
	\[
	\lambda_p=\eta^{k-1},
	\qquad
	\lambda_q=\lambda_r=\eta^k.
	\]
	Therefore
	\[
	\left.A_{m,\eta}\right|_{\{p,q,r\}}
	=
	\begin{pmatrix}
		1&-1&-1\\
		-\eta^2&1&-\frac12+4\eta^2\\
		-\eta^2&-\frac12+4\eta^2&1
	\end{pmatrix}.
	\]
	Finally, we note that if $m=2k$ is even, then $ \lambda_{m-1}=\eta^{k-1},
	\quad \lambda_m=\eta^k$, and the terminal block becomes
	\[
	\left.A_{m,\eta}\right|_{\{m-1,m\}}
	=
	\begin{pmatrix}
		1&-1\\
		-\eta^2&1
	\end{pmatrix},
	\]
	whereas
	\[
	\left.C_m\right|_{\{m-1,m\}}
	=
	\begin{pmatrix}
		1&-1\\
		0&1
	\end{pmatrix}.
	\]
	Thus the discrepancy on the terminal block is at most
	$\eta^2$. Together with the estimates for complete $3\times 3$ blocks, we have	
	\[
	\normmax{A_{m,\eta}-C_m}\leq4\eta^2.
	\]
\end{proof}

\section{Proof of Theorem~\ref{thm:main}}
\label{section proof of main theorem}
\begin{proof}[Proof of Theorem~\ref{thm:main}]
	Fix $m\geq2$. By Lemma~\ref{lem:model-trajectory} and Lemma~\ref{lemma 1}, the matrix $C_m$ admits a genuine $C_m$-trajectory with $N_m$ collisions. By Lemma~\ref{lemma 2}, there exists $\delta_m>0$ such that every admissible matrix $A$ satisfying
	\[
	\normmax{A-C_m}<\delta_m
	\]
	admits an $A$-trajectory with $N_m$ collisions.
	
	Choose $\eta_m\in(0,1/10]$ sufficiently small that
	\[
	4\eta_m^2<\delta_m.
	\]
	Proposition~\ref{prop:convergence} implies that $A_{m,\eta_m}$ admits a trajectory with $N_m$ collisions. Since
	\[
	B_m(\eta_m)=A_{m,\eta_m}^{\lambda},
	\]
	Lemma~\ref{lemma 3} shows that the symmetric matrix $B_m(\eta_m)$ admits a trajectory with the same number of collisions.
	
	By Lemmas~\ref{lem:packing} and~\ref{lem:independence}, $B_m(\eta_m)$ is the Gram matrix of the linearly independent wall normals of the tangent cone at a planar configuration of $m+1$ identical hard disks. The corresponding $B_m(\eta_m)$-trajectory is therefore the wall-coordinate representation of a billiard trajectory in that tangent cone. Applying Lemma~\ref{lemma 4}, we obtain a genuine trajectory of $m+1$ identical hard disks in $\R^2$ with at least $N_m$ collisions. Hence
	\[
	\MaxCol(m+1,2)\geq N_m.
	\]
	
	Set $n=m+1$. If $n=2k+1$, then
	\[
	N_{n-1}=6\cdot2^k-3k-6\geq2^k.
	\]
	If $n=2k$ with $k\geq2$, then
	\[
	N_{n-1}=\frac{9}{2}\cdot 2^k-3k-5\geq2^k.
	\]
	Thus
	\[
	N_{n-1}\geq2^{\lfloor n/2\rfloor}
	\]
	for every $n\geq3$, and the theorem follows.
\end{proof}

\begin{remark}
	\label{remark finite time interval}
	An equivalent formulation of Lemma~\ref{lemma 4} \cite[Lemma~2.2]{BI} can be used for numerical evaluation: Let $q$ be a coordinate in configuration space of an $n$-ball system in $\RR^d$, and suppose that $ \gamma:(a,b)\longrightarrow K $ is a nonsingular billiard trajectory with $N$ collisions at times $ t_1<\cdots<t_N$.
	Then, for every sufficiently small $\epsilon>0$, there exists a
	billiard trajectory $\widetilde\gamma_\epsilon$ in $B_{n,d}$ with
	$N$ collisions occurring in the same order, such that 
	\[
	\widetilde\gamma_\epsilon(t) = q + \epsilon \gamma(t) + o(\epsilon).
	\]
	Moreover, for every $k=1,...,N$, the corresponding collision moments of the actual trajectory,
	$\widetilde t_k(\epsilon)$ converges to $ t_k $ as $\epsilon\to0$.
	In particular, $ \MaxCol(n,d)\geq N$.
	
	In the proof of Lemma~\ref{lem:model-trajectory},
	the spacing at the $k$-th induction step is
	$\beta_k=2^{-(k-1)}$, and the new outermost
	collision times extend the previous interval
	by $3\beta_k/4$ on each side. Hence the
	limiting half-length is bounded by $ \sum_{k\ge1}\frac{3}{4}\cdot 2^{-(k-1)}
	=\frac32$. Therefore, all collisions of the generalized
	model trajectory occur within an interval
	of length at most $3$.
	
	By choosing the perturbations sufficiently small,
	we may arrange that the corresponding genuine
	tangent-cone trajectory has collision times
	$t_1<\cdots<t_{N_m}$ satisfying $ t_{N_m}-t_1<3 $.
	
	Let $\widetilde t_j(\epsilon)$ denote the
	corresponding collision times of the actual
	trajectory furnished by Lemma~\ref{lemma 4}.
	Then $	\widetilde t_j(\epsilon)\longrightarrow t_j
	\quad\text{as }\epsilon\to0 $.
	
	After reparametrizing the actual trajectory by
	$t\mapsto t_0+t/\epsilon$, where $t_0<t_1$,
	the $j$-th physical collision moment become $ \epsilon\bigl(\widetilde t_j(\epsilon)-t_0\bigr) $. Consequently, the physical collision period
	satisfies
	\[
	T_{\mathrm{phys}}(\epsilon)
	=
	\epsilon(t_{N_m}-t_1)+o(\epsilon)
	<3\epsilon
	\]
	for all sufficiently small $\epsilon>0$. Hence collisions occur in a finite time period uniformly bounded above regardless of $n$.
\end{remark}

%


\begin{appendices}
\newpage	
	\section{Actual trajectory of four-disk system}
	We give explicit initial positions and velocities
	for a four-disk trajectory realizing the collision
	sequence of Example~\ref{example four disks}.
	\begin{example}[Four disks (continued)]
		\label{example four disks continued}
		Recall that $s=\sqrt{1-c^2}$. Take the reference contact configuration
		\[
		\bar q_1=(0,1),\qquad
		\bar q_2=(0,0),\qquad
		\bar q_3=(-s,-c),\qquad
		\bar q_4=(s,-c),
		\]
		and enumerate its three contacts by
		\[
		e_1=\{q_1,q_2\},\qquad
		e_2=\{q_2,q_3\},\qquad
		e_3=\{q_2,q_4\}.
		\]
		Recall Lemma~\ref{lemma 4}.
		Given $	f(-1)=\left(1,\frac34,\frac14\right),
		\quad f'(-1)=(-1,-1,-1)$ and taking
		$
		(\lambda_1,\lambda_2,\lambda_3)
		=\left(1,\frac c2,\frac c2\right)$, we introduce a parameter $\varepsilon>0$
		controlling the displacement of the initial
		positions from the reference contact configuration.
		We prescribe
		\[
		\begin{aligned}
			q_1(0)-q_2(0)
			&=\bigl(1+\varepsilon\lambda_1f_1(-1)\bigr)
			(\bar q_1-\bar q_2),\\
			q_3(0)-q_2(0)
			&=\bigl(1+\varepsilon\lambda_2f_2(-1)\bigr)
			(\bar q_3-\bar q_2),\\
			q_4(0)-q_2(0)
			&=\bigl(1+\varepsilon\lambda_3f_3(-1)\bigr)
			(\bar q_4-\bar q_2),
		\end{aligned}
		\]
		and
		\[
		\begin{aligned}
			\dot q_1(0)-\dot q_2(0)
			&=\lambda_1f_1'(-1)(\bar q_1-\bar q_2),\\
			\dot q_3(0)-\dot q_2(0)
			&=\lambda_2f_2'(-1)(\bar q_3-\bar q_2),\\
			\dot q_4(0)-\dot q_2(0)
			&=\lambda_3f_3'(-1)(\bar q_4-\bar q_2).
		\end{aligned}
		\]
		Take $q_2(0)=\bar q_2$ and $\dot q_2(0)=0$. We have
		\[
		\begin{aligned}
			q_1(0)&=(1+\varepsilon)\bar q_1,
			&
			\dot q_1(0)&=(0,-1),\\
			q_2(0)&=\bar q_2,
			&
			\dot q_2(0)&=(0,0),\\
			q_3(0)&=\left(1+\frac{3\varepsilon c}{8}\right)\bar q_3,
			&
			\dot q_3(0)&=\frac c2(s,c),\\
			q_4(0)&=\left(1+\frac{\varepsilon c}{8}\right)\bar q_4,
			&
			\dot q_4(0)&=\frac c2(-s,c).
		\end{aligned}
		\]
		The corresponding normalized wall coordinates are
		$\varepsilon\lambda_i f_i(-1)/\sqrt2$ and their derivatives are
		$\lambda_i f_i'(-1)/\sqrt2$; the common factor $1/\sqrt2$ is immaterial.
		
		Note that with $m=3$ and $ \eta=\frac c2,
		\quad
		(\lambda_1,\lambda_2,\lambda_3)
		=
		\left(1,\frac c2,\frac c2\right) $,
		we have
		\[
		C_3=
		\begin{pmatrix}
			1&-1&-1\\
			0&1&-\frac12\\
			0&-\frac12&1
		\end{pmatrix},
		\]
		\[
		B_3(\frac{c}{2})=
		\begin{pmatrix}
			1&-\frac c2&-\frac c2\\
			-\frac c2&1&-\frac12+c^2\\
			-\frac c2&-\frac12+c^2&1
		\end{pmatrix},
		\]
		and
		\[
		A_{3,\frac{c}{2}}=
		\begin{pmatrix}
			1&-1&-1\\
			-\frac{c^2}{4}&1&-\frac12+c^2\\
			-\frac{c^2}{4}&-\frac12+c^2&1
		\end{pmatrix}.
		\]
		It is not hard to see that $
		B_3(\frac{c}{2})=A_{3,\frac{c}{2}}^\lambda $ with $	(A_{3,\frac{c}{2}}^\lambda)_{ij}
		=
		\frac{\lambda_j}{\lambda_i}a_{ij}
		$, and $A_{3,\frac{c}{2}}$ converges to $C_3$ as $c\to 0$.
		
		For every sufficiently small $c>0$, these initial data realize the desired collision sequence in Example~\ref{example four disks} whenever $\varepsilon>0$ is sufficiently small.
	\end{example}
\end{appendices}

\end{document}